\documentclass[11pt,a4paper,reqno]{amsart}

\usepackage[linktocpage]{hyperref}
\hypersetup{
  colorlinks   = true, 
  urlcolor     = blue, 
  linkcolor    = Purple, 
  citecolor   = red 
}
\usepackage{changepage}
\usepackage{amsmath}
\usepackage{amsthm}
\usepackage{amsfonts}
\usepackage{amssymb}
\usepackage{array}
\usepackage[margin=2.7cm]{geometry}

\usepackage{mathtools}
\usepackage{soul}
\usepackage{enumitem}
\usepackage{hhline}
\usepackage{multirow} 

\makeatletter
\renewcommand{\tocsection}[3]{%
  \indentlabel{\@ifnotempty{#2}{\bfseries\ignorespaces#1 #2\quad}}\bfseries#3}
\renewcommand{\tocsubsection}[3]{%
  \indentlabel{\@ifnotempty{#2}{\ignorespaces#1 #2\quad}}#3}

\newcommand\@dotsep{4.5}
\def\@tocline#1#2#3#4#5#6#7{\relax
  \ifnum #1>\c@tocdepth 
  \else
    \par \addpenalty\@secpenalty\addvspace{#2}%
    \begingroup \hyphenpenalty\@M
    \@ifempty{#4}{%
      \@tempdima\csname r@tocindent\number#1\endcsname\relax
    }{%
      \@tempdima#4\relax
    }%
    \parindent\z@ \leftskip#3\relax \advance\leftskip\@tempdima\relax
    \rightskip\@pnumwidth plus1em \parfillskip-\@pnumwidth
    #5\leavevmode\hskip-\@tempdima{#6}\nobreak
    \leaders\hbox{$\m@th\mkern \@dotsep mu\hbox{.}\mkern \@dotsep mu$}\hfill
    \nobreak
    \hbox to\@pnumwidth{\@tocpagenum{\ifnum#1=1\bfseries\fi#7}}\par
    \nobreak
    \endgroup
  \fi}
\AtBeginDocument{%
\expandafter\renewcommand\csname r@tocindent0\endcsname{0pt}
}
\def\l@subsection{\@tocline{2}{0pt}{2.5pc}{5pc}{}}
\makeatother

\usepackage{etoolbox}
\makeatletter
\patchcmd{\@setaddresses}{\indent}{\noindent}{}{}
\patchcmd{\@setaddresses}{\indent}{\noindent}{}{}
\patchcmd{\@setaddresses}{\indent}{\noindent}{}{}
\patchcmd{\@setaddresses}{\indent}{\noindent}{}{}
\makeatother

\makeatletter
\newcommand{\mylabel}[2]{#2\def\@currentlabel{#2}\label{#1}}
\makeatother

\newcommand{\C}{\mathcal{C}}
\newcommand{\A}{\mathcal{A}}
\newcommand{\mU}{\mathcal{U}}
\newcommand{\mV}{\mathcal{V}}
\newcommand{\mC}{\mathcal{C}}

\newcommand{\mB}{\mathcal{B}}

\newcommand{\mA}{\mathcal{A}}

\newcommand{\F}{\mathbb{F}}

\newcommand{\N}{\mathbb{N}}

\newcommand{\Fq}{\F_{q}}
\newcommand{\Fm}{\F_{q^m}}

\newcommand{\Fmk}{[n,k]_{q^m/q}}
\newcommand{\Fmkd}{[n,k,d]_{q^m/q}}

\DeclareMathOperator{\GL}{GL}

\DeclareMathOperator{\Gal}{Gal}
\DeclareMathOperator{\rk}{rk}

\DeclareMathOperator{\wt}{wt}

\DeclareMathOperator{\Tr}{Tr}

\DeclareMathOperator{\dd}{d}
\DeclareMathOperator{\Rdef}{Rdef}
\DeclareMathOperator{\rowsp}{rowsp}

\theoremstyle{definition}
\newtheorem{theorem}{Theorem}[section]

\newtheorem{proposition}[theorem]{Proposition}
\newtheorem{corollary}[theorem]{Corollary}

\newtheorem{definition}[theorem]{Definition}
\newtheorem{example}[theorem]{Example}
\newtheorem{remark}[theorem]{Remark}

 \usepackage[dvipsnames,table]{xcolor}
 \usepackage{pagecolor}
 \definecolor{light-gray}{gray}{0.90}

\usepackage{tikz-cd} 

\usepackage{tikz}
\usetikzlibrary{arrows.meta}
\usepackage{lscape}
\usepackage{arydshln}

\newcommand{\St}{\,:\,}

\usetikzlibrary{shapes.geometric,quotes}

\title[Evasive Subspaces, Generalized Rank Weights and near MRD codes]{Evasive Subspaces, Generalized Rank Weights \\ and near MRD codes}
\date{}
\author[G. Marino]{Giuseppe Marino}
\address{Giuseppe Marino, \textnormal{Department of Mathematics and Applications ``R. Caccioppoli'',
	University of Naples ``Federico II'',
	Via Cintia, Monte S.Angelo I-80126 Naples, Italy}}
\email{giuseppe.marino@unina.it}
\author[A. Neri]{Alessandro Neri}
\address{Alessandro Neri, \textnormal{Max-Planck-Institute for Mathematics in the Sciences, Inselstraße 22, 04103 Leipzig, Germany}}
\email{alessandro.neri@mis.mpg.de}
\author[R. Trombetti]{Rocco Trombetti}
\address{Rocco Trombetti, \textnormal{Department of Mathematics and Applications ``R. Caccioppoli'',
	University of Naples ``Federico II'',
	Via Cintia, Monte S.Angelo I-80126 Naples, Italy}}
\email{rtrombet@unina.it}

\begin{document}

\maketitle

\begin{abstract}
    We revisit and extend the connections between $\Fm$-linear rank-metric codes and evasive $\Fq$-subspaces of $\Fm^k$. We give a unifying  framework in which we prove in an elementary way how the parameters of a rank-metric code are related to special geometric properties of the associated evasive subspace, with a particular focus on the generalized rank weights. In this way, we can also provide alternative and very short proofs of known results on  scattered subspaces. We then use this simplified point of view in order to get a geometric characterization of near MRD codes and a clear bound on their maximal length. Finally we connect the theory of quasi-MRD codes with $h$-scattered subspaces of maximum dimension, extending to all the parameters sets the already known results on MRD codes. 
\end{abstract}

\tableofcontents

\section{Introduction}
\noindent \textbf{Context.} Rank-metric codes were introduced by Delsarte in 1978 \cite{delsarte1978bilinear} and independently by Gabidulin in 1985 \cite{gabidulin1985theory}, but only recently the interest of  researchers increased exponentially due to their applications to network coding \cite{silva2008rank}. However, many other applications have emerged, such as crisscross error correction \cite{roth1991maximum}, code-based cryptography \cite{gabidulin1991ideals}, distributed storage \cite{rawat2013optimal} and space-time coding \cite{liu2002rank}, to name a few. 

Due to many of these applications, rank-metric codes are essentially always considered over finite fields, although they have been studied also over infinite fields \cite{augot2013rank,augot2021rankmetric}, with some interesting implications over Kummer extensions of the rationals \cite{robert2015new}. However, rank-metric codes possess interesting algebraic, geometric and combinatorial properties, which make them fascinating also for pure mathematicians. Indeed, their mathematical study has been progressing in parallel, and connections between rank-metric codes and other mathematical objects have emerged in the last decade. These connections include, but are not limited to, finite semifields \cite{sheekey2016new}, linear sets in finite geometry \cite{lunardon2017mrd}, $q$-analogues in combinatorics \cite{gorla2020rank} and tensors \cite{roth1996tensor}. 

Among the combinatorial and geometric features of a rank-metric code, it is worth to mention its \emph{generalized rank weights}. These invariants of a rank-metric code  generalize the notion of minimum rank distance, and form a strictly increasing sequence of integers (\emph{monotonicity}). Another interesting combinatorial property is that the generalized rank weights of an $\Fm$-linear rank-metric  code completely determine the generalized rank weights of the dual code (\emph{Wei-type duality}). However, generalized rank weights have shown to be important also from an application point of view. In \cite{kurihara2015relative}, it was shown that the security performance and the error correction capability of secure network coding are strongly dependent on the generalized rank weights of the underlying rank-metric code.

As it happens for linear codes in the Hamming metric, also rank-metric codes have a geometric equivalent interpretation. This point of view has been introduced in  detail in \cite{sheekey2019scattered,randrianarisoa2020geometric}, even though  partial correspondences were already highlighted in \cite{sheekey2016new} in some special cases. Formally, to a $k$-dimensional (as $\Fm$-vector space) rank-metric code $\C\subseteq \Fm^n$, one can associate an $n$-dimensional $\Fq$-subspace $\mU$ of $\Fm^k$, called \emph{$q$-system}. Such a space $\mU$ is  defined up to equivalence, and  can be  obtained by taking the $\Fq$-span of the columns of a generator matrix of $\C$. This provides a sort of dual interpretation, where the metric properties of the code can be translated into geometric properties of $\mU$. 

In this geometric direction, $q$-systems have been intensively investigated in recent years, in particular for what concerns $(h,r)$-\emph{evasive subspaces},
 which are $q$-systems $\mU$ with special intersection properties: Every $h$-dimensional $\Fm$-subspace of $\Fm^k$ intersects $\mU$ in an $\Fq$-subspace of $\Fq$-dimension at most $r$.
 Evasive subspaces are generalizations of \emph{$h$-scattered subspaces}, introduced and studied in \cite{blokhuis2000scattered} when $h=1$ and in \cite{csajbok2021generalising} for the general case.
 
 \bigskip
 
\noindent \textbf{Overview.}  In this paper, we first provide a clear and concise theory connecting the study of $q$-systems and the study of rank-metric codes. The main tool is to analyze them including also their {generalized rank weights}, which allow us to give a complete understanding of the parameters of a rank-metric codes and the evasive properties of the associated $q$-system. This is the case of Theorem \ref{thm:evasive_generalizedweights}, which is the main result unifying the theory of rank-metric codes with the one of $q$-systems. This can be read as follows: 
\medskip

\noindent 
\begin{center}Let $\C\subseteq \Fm^n$ be a $k$-dimensional code and $\mU$ be its associated $q$-system. Then 
$\mU$ is $(h,r)$-evasive \emph{if and only if} the $(k-h)$-th generalized rank weight of $\C$ is at least $n-r$ \emph{if and only if} the $(r-h+1)$-th generalized rank weight of $\C^\perp$ is at least $r+2$. 
\end{center}

\medskip

The rest of the paper is dedicated to use this connection in both directions exploiting the interplay between geometry and coding theory: we first obtain geometric results from well-known coding theoretic properties of rank metric codes  and their generalized rank weights, and then we use the geometry of $q$-systems to study new classes of rank-metric codes.
More specifically, we first derive a series of consequences and rediscoveries of geometric results on $h$-scattered subspaces and their links with MRD codes with the aid of Theorem \ref{thm:evasive_generalizedweights}. These are some of the main results presented in \cite{csajbok2021generalising,zini2021scattered,bartoli2021evasive}; see Theorem \ref{thm:hscattered_MRD} and Corollaries \ref{cor:hscattered_bound}, \ref{cor:evasive_bounds} and \ref{cor:maxscatterd_hyperplane_intersection} among the others.
Although these results are already known, the new proofs are significantly shorter and simplified. This has a lot to do with the Wei-type duality of the generalized rank weights of a rank-metric code and its dual.

At a later stage, we focus on near MRD codes: they are $k$-dimensional codes $\C\subseteq \Fm^n$  which are very close to be MRD, having minimum rank distance equal to $n-k$ and meeting the Singleton bound for all the remaining generalized rank weights. Here, we use the opposite direction of the connection,  exploiting the geometric perspective to characterize this family of codes and to derive constraints on their parameters. More precisely, in Theorem \ref{thm:nearMRD_characterization} we show that 
the parameters of a near MRD code satisfy
$$ m \geq k \quad \mbox{ and }  \quad   n \leq \begin{cases} m+2 & \mbox{ if } m=2k-2, \\
 m+1 & \mbox{ if }  m \neq 2k-2,
 \end{cases} $$
characterizing near MRD codes of length $m+2$.

We then conclude by highlighting the connection between quasi-MRD codes introduced in \cite{de2018weight} with the so-called \emph{quasi-maximum $h$-scattered subspaces}. The latter are $h$-scattered $q$-systems of maximum $\Fq$-dimension $n=\lfloor\frac{km}{h+1}\rfloor$, when $(h+1)$ does not divide $km$. This connection is explained in Theorem \ref{thm:quasimaximum_quasiMRD} and it extends the one between maximum $h$-scattered subspaces and MRD codes provided in \cite{zini2021scattered}.

\bigskip

\noindent \textbf{Organization.} The paper is structured as follows. Section \ref{sec:rankmetric} contains the preliminary notions and results about rank-metric codes, with a focus on generalized rank weights, and their geometric interpretation as $q$-systems. In Section \ref{sec:evasive_genweights} we introduce evasive subspaces, and show the main correspondence result (Theorem \ref{thm:evasive_generalizedweights}) between evasive subspaces and rank-metric codes with lower bounded generalized rank weights. Section \ref{sec:consequences} is dedicated to illustrate how we can derive several important geometric properties from this correspondence with substantially simplified arguments. The class of near MRD codes is deeply investigated in Section \ref{sec:nearMRD}, where we also derive bounds on their parameters. Finally, Section \ref{sec:quasi_quasi} explains the link between quasi-MRD codes and quasi-maximum $h$-scattered subspaces.
\bigskip

\noindent \textbf{Notation.} In the following, we fix $\Fq$ to be a finite field with $q$ elements, and $\Fm$ will be the degree $m$ extension of $\Fq$. Moreover, we use $k$ and $n$ to denote two positive integers with $k \leq n$.
Throughout the paper, vectors will be always considered as \textit{row vectors}, unless otherwise specified.

\section{Rank-metric codes}\label{sec:rankmetric}
Rank-metric codes over finite fields can be equivalently represented as sets of $n\times m$ matrices over $\Fq$ or as sets of vectors of length $n$ over $\Fm$. This is because the choice of an $\Fq$-basis of $\Fm$ only affects the matrix representation of vectors in $\Fm^n$, but not its rank. More precisely, different choices of $\Fq$-bases of $\Fm$ lead to equivalent representations of codes in $\Fq^{n\times m}$. However, things change when one considers the natural linearity of the underlying field.  In this paper, we will only consider rank-metric codes as $\Fm$-linear subspaces of $\Fm^n$.

\subsection{Definitions and properties}
On the space of vectors $\Fm^n$ over $\Fm$, we define the \textbf{rank weight} as 
 $$\wt_{\rk}(v)=\dim_{\Fq}(\langle v_1,\ldots,v_n\rangle_{\Fq}), \qquad \mbox{ for any } v=(v_1,\ldots,v_n) \in \Fm^n,$$
 and the \textbf{rank distance} between two vectors is defined as $\dd_{\rk}(u,v)\coloneqq\wt_{\rk}(u-v)$.

\begin{definition}
 An $\Fmkd$ \textbf{(rank-metric) code} $\mC$ is a $k$-dimensional $\Fm$-subspace of $\Fm^n$ equipped with the rank distance. The parameter $d$ is called the \textbf{minimum rank distance} and it is given by
 \begin{align*} d\coloneqq\dd_{\rk}(\mC)&= \min \{\dd_{\rk}(u,v) \St u,v\in\mC, u \neq v \} \\&=\min\{\wt_{\rk}(v) \St v \in \mC, v\neq 0 \}.
 \end{align*}
 When the parameter $d$ is not known or relevant, we will simply write that $\C$ is an $\Fmk$ code.
 
 A \textbf{generator matrix} for $\C$ is a matrix $G \in \Fm^{k\times n}$  such that 
 $$\C=\big\{ vG \St v \in \Fm^k \big\}. $$
\end{definition}

The following is a classical result relating the parameters of a rank-metric code. 

\begin{theorem}[Singleton Bound \textnormal{\cite{delsarte1978bilinear}}]
 Let $\C$ be an $\Fmkd$ code. Then
 \begin{equation}\label{eq:singleton} mk\leq \min\{m(n-d+1), n(m-d+1)\}.\end{equation}
\end{theorem}

A code $\C$ is said to be \textbf{maximum rank distance (MRD)} if its parameters meet the bound in \eqref{eq:singleton} with equality. 

\begin{remark}[Existence of $\Fmk$ MRD codes]
Existence and constructions of $\Fmk$ MRD codes are known for many parameters sets. The first construction of $\Fmk$ MRD codes was provided already in the seminal papers by Delsarte \cite{delsarte1978bilinear} and Gabidulin \cite{gabidulin1985theory}, whenever $n\le m$. These codes are known today as \textbf{Delsarte-Gabidulin codes}. When $n\le m$, another construction was also provided more recently by Sheekey in \cite{sheekey2016new}. Thus, $\Fm$-linear  MRD codes exist for any parameter sets, when $n \leq m$.

For the case $n>m$, additional families can be obtained when $n=tm$ and $k=tk'$ by doing the direct sum of $t$ copies of an $[m,k',m-k'+1]_{q^m/q}$ MRD code. Recently, new constructions have been found geometrically in \cite{bartoli2018maximum,csajbok2017maximum}, for  any $k$, $m$ even and $n=\frac{mk}{2}$, giving $[\frac{mk}{2},k,m-1]_{q^m/q}$ MRD codes. Together with their dual, which are $[\frac{mk}{2},\frac{(m-2)k}{2},3]_{q^m/q}$ MRD codes, these are the only general constructions of $\Fm$-linear MRD codes known so far. 
We conclude mentioning two recent special constructions. A family of $[8,4,3]_{q^4/q}$ MRD codes was found in \cite{BMN}, when $q$ is an odd power of $2$, which cannot be obtained as the direct sum of two $[4,2,3]_{q^4/q}$ MRD codes. In \cite{BMN2}, a family of  $[8,4,4]_{q^6/q}$ MRD codes was constructed for $q=2^h$ and $h \equiv 1 \pmod 6$.
\end{remark}

\begin{definition}\label{def:dualcode}
 Let $\C$ be an  $\Fmkd$  code. Then its \textbf{dual code} $\C^\perp$ is given by
 $$ \C^\perp\coloneqq\left\{ v \in \Fm^n \St cv^\top=0, \, \forall\, c \in \C\right\},$$ where the symbol $v^{\top}$ denotes here the transpose of $v$.
\end{definition}

In this setting, we can consider the equivalence of codes. Let us only take those which are linear. In this case two $\Fmk$ codes $\mC_1$ and $\mC_2$ are \textbf{equivalent} if and only if there exist $A\in \GL(n,q)$ such that 
$$\C_1=\C_2\cdot A=\{vA \St v \in \mC_2\}.$$

We conclude with the notion of degeneracy of a rank-metric code. This can be defined in many equivalent ways. For $\Fmk$ codes we report two equivalent formulations taken from \cite{alfarano2021linear}.

\begin{proposition}[\textnormal{\cite[Proposition 3.2]{alfarano2021linear}}]\label{prop:nondegenerate}
Let $\C$ be an $\Fmk$ code. The following are equivalent.
\begin{enumerate}
    \item The $\F_q$-span of the columns of any generator matrix of $G$ has $\Fq$-dimension $n$.\footnote{Note that the $\Fq$-span of the columns of a  generator matrix of a code $\C$ depends on the choice of the generator matrix. However, its $\Fq$-dimension does not.}
    \item $\dd_{\rk}(\mC^\perp) \ge 2$.
\end{enumerate}
\end{proposition}

\begin{definition}
 An $\Fmk$ code is \textbf{nondegenerate} if it satisfies any of the equivalent conditions in Proposition \ref{prop:nondegenerate}.
\end{definition}

\subsection{Generalized rank weights}

Generalized rank weights have been first introduced and studied by Kurihara-Matsumoto-Uyematsu \cite{kurihara2012new,kurihara2015relative}, Oggier-Sboui \cite{oggier2012existence} and Ducoat \cite{ducoat2015generalized}. They are the counterpart of generalized weights in Hamming metric and are of great  interest due to their combinatorial properties \cite{ravagnani2016generalized,jurrius2017defining,ghorpade2020polymatroid} and their applications to network coding \cite{martinez2016similarities,martinez2017relative}. We recall here the definition and their properties. 

Let $\mA \subseteq \Fm^n$ and $\theta \in \Gal(\Fm/\Fq)$, and denote by $\theta(\mA)$ the image of $\mA$ under the componentwise map $\theta$, that is 
$$\theta(\mA)\coloneqq \{(\theta(a_1),\ldots,\theta(a_n)) \St(a_1,\ldots,a_n) \in \mA\}.$$
We say that $\mA$ is \textbf{Galois closed} if  $\theta(\mA)=\mA$ for every $\theta \in \Gal(\Fm/\Fq)$. Observe that it is enough that $\theta(\mA)=\mA$ for a generator of $\Gal(\Fm/\Fq)$. For instance, we can just consider $\theta$ to be the $q$-Frobenius automorphism. We denote by $\Lambda_q(n,m)$ the set of Galois closed $\Fm$-subspaces of $\Fm^n$.

\begin{definition}
 Let $\C$ be an $\Fmkd$ code, and let $1\leq s \leq k$. The \textbf{$s$-th generalized rank weight} of $\C$ is the integer
 $$ \dd_{\rk,s}(\C)\coloneqq\min \{\dim_{\Fm}(\mA) \St \mA \in  \Lambda_q(n,m), \dim_{\Fm}(\mA\cap \C) \geq s\}.$$
\end{definition}

The following result recaps the two most important properties of generalized rank weights: the monotonicity and the Wei-type duality.

\begin{proposition}[\textnormal{\cite{kurihara2015relative,ducoat2015generalized}}]\label{prop:gen_weights_properties}
 Let $\mC$ be a nondegenerate $\Fmkd$ code. Then
 \begin{enumerate}
     \item (Monotonicity) $d=\dd_{\rk,1}(\mC)<\dd_{\rk,2}(\mC)<\ldots<\dd_{\rk,k}(\mC)=n$.
     \item (Wei-type duality) $\{\dd_{\rk,i}(\mC) \St 1\leq i \leq k\}\cup \{n+1-\dd_{\rk,i}(\mC^\perp) \St 1\leq i \leq n- k\}=\{1,\ldots,n\}$.
 \end{enumerate}
\end{proposition}

Other important results are about bounds on the generalized rank weights of an $\Fmk$ code. We recap these bounds in the following proposition.

\begin{proposition}[Bounds \textnormal{\cite{martinez2016similarities}}]\label{prop:bounds_genweights}
 Let $\C$ be an  $\Fmkd$ code and let $1\le s\le k$. Then
 \begin{equation}\label{eq:singbound_genweights}\dd_{\rk,s}(\C) \leq \min\Big\{n-k+s, sm, \frac{m}{n}(n-k)+m(s-1)+1 \Big\}.\end{equation}
\end{proposition}

\begin{remark} Here some remarks about Proposition \ref{prop:bounds_genweights}.
\begin{enumerate}
\item Putting $s=1$ in \eqref{eq:singbound_genweights}, we recover the Singleton-like bound of \eqref{eq:singleton}.
\item When $m\geq n$, then the right hand side of \eqref{eq:singbound_genweights} is always equal to $n-k+s$. 
\item The converse is not true. The right hand side of the bound of \eqref{eq:singbound_genweights} can be  equal to $n-k+s$ also when $m<n$. We will see some examples in Section \ref{sec:nearMRD}. 
\end{enumerate}
\end{remark}

\begin{definition}
 Let $s$ be a positive integer.  An $\Fmk$ code $\C$ is  $s$-MRD if 
 $d_{\rk,s}(\C)=n-k+s$.
\end{definition}

Notice that if an $\Fmk$ code $\C$ is $s$-MRD for some $s$, then it is also $s'$-MRD for every $s\leq s'\leq k$.

\begin{remark}\label{rem:MRDvs1MRD}
Observe that a $1$-MRD code is also an MRD code. {However, an MRD code is not necessarily $1$-MRD. This is because we defined MRD codes to be those meeting the minimum of the two bounds in \eqref{eq:singleton} with equality.} Indeed, for us the parameters $n$ and $m$ are not interchangeable, while there is no difference between them when considering $\Fq$-linear rank-metric codes. This is  due to their representation as $\Fq$-subspaces of  $\Fq^{n\times m}$ and to the fact that here the transposition of matrices is an $\Fq$-linear isometry.  
\end{remark}

\begin{remark}
According to Proposition \ref{prop:bounds_genweights}, an $\Fmk$ $s$-MRD code may exist only if 
$$n-k+s \leq \min \Big\{sm,\frac{m}{n}(n-k)+m(s-1)+1 \Big\}.$$
For $1$-MRD codes, this happens only when $m\geq n$, and it is well known that this condition is also sufficient for the existence of such codes; see \cite{delsarte1978bilinear,gabidulin1985theory}. However, for $s\geq 2$, there is very little known about $s$-MRD codes which are not $1$-MRD. In principle, even the condition $m\geq n$ might not be necessary for the existence of an $s$-MRD code when $s\geq 2$. 
This is actually the case sometimes, as we will see later in Corollary \ref{cor:sMRD_exist}.
\end{remark}

\subsection{The geometry of rank-metric codes}
In this subsection we recall the geometric point of view of $\Fm$-linear rank-metric codes via $q$-systems. This aspect was first observed in  \cite{sheekey2019scattered} and independently also in \cite{randrianarisoa2020geometric}. We will use the notation of \cite{alfarano2021linear}, where the connection between rank-metric codes and $q$-systems has been clarified in detail.

\begin{definition}
 An \textbf{$\Fmkd$ system} $\mU$ is an $\Fq$-subspace of $\Fm^k$ of dimension $n$ such that $\langle \mU\rangle_{\Fm}=\Fm^k$. The integer $d$ is defined as
 $$ d\coloneqq n-\max\{\dim_{\Fq}(H\cap \mU) \St H\subseteq \Fm^k,  \dim_{\Fm}(H)=k-1 \}. $$
 For brevity, when the parameters are not relevant, we will simply say that $\mU$ is a \textbf{$q$-system}.
 
 Two  $\Fmkd$ systems $\mU_1$ and $\mU_2$ are said to be \textbf{(linearly) equivalent} if there exists $M \in \GL(k,q^m)$ such that $\mU_1= \mU_2 \cdot M$.
\end{definition}

Let $\mathfrak U(n,k,d)_{q^m/q}$ denote the set of equivalence classes of $\Fmkd$ systems, and let $\mathfrak C(n,k,d)_{q^m/q}$ denote the set of equivalence classes of nondegenerate $\Fmkd$ codes. Then, define the maps
$$\begin{array}{rccc}\Phi: & \mathfrak C(n,k,d)_{q^m/q} &\longrightarrow &\mathfrak U(n,k,d)_{q^m/q} \\
& [\rowsp( u_1^\top \mid \ldots \mid u_n^\top)] & \longmapsto & [\langle u_1, \ldots, u_n\rangle_{\Fq}] \end{array} $$

$$\begin{array}{rccc}\Psi: & \mathfrak U(n,k,d)_{q^m/q} &\longrightarrow &\mathfrak C(n,k,d)_{q^m/q} \\
& [\langle u_1, \ldots, u_n\rangle_{\Fq}] & \longmapsto & [\rowsp( u_1^\top \mid \ldots \mid u_n^\top)] \end{array} $$

\begin{theorem}[\textnormal{\cite{randrianarisoa2020geometric}}]\label{thm:correspondence_codes_systems}
 The maps $\Phi$ and $\Psi$ are well-defined and they are one the inverse of each other. Hence, they define a 1-to-1 correspondence between equivalence classes of $\Fmkd$ codes and equivalence classes of $\Fmkd$ systems. Moreover, under this correspondence that associates a $\Fmkd$ code $\C$ to an $\Fmkd$ system $\mU$, codewords of   $\C$ of rank weight $w$ correspond to $\Fm$-hyperplanes $H$ of $\Fm^k$ with $\dim_{\Fq}(H\cap \mU)=n-w$.
\end{theorem}

As stated in Theorem \ref{thm:correspondence_codes_systems}, the correspondence is not only between codes and systems, but it also induces a relation between codewords and hyperplanes. 

For a given  $\Fmkd$ system $\mU$ choose any $\Fq$-basis $(g_1,\ldots g_n)$, where the $g_i$'s belong to $\Fm^k$. Now, put them as  columns of a matrix $G=(g_1 \mid \ldots \mid g_n)\in \Fm^{k \times n}$. The code corresponding to $\mU$ is $\C=\rowsp(G)$. 
At this point, any nonzero codeword of $\C$ is of the form $vG$, for any $v\in\Fm^k$. Then, we have
\begin{equation}\label{eq:rank_weight}\wt_{\rk}(vG)=n-\dim_{\Fq}((\langle v\rangle_{\Fm})^\perp\cap \mU).\end{equation}

\begin{remark}
Observe that in the special case that $n=m$, we have that $[n,k]_{q^n/q}$ codes can be equivalently represented as $\F_{q^n}$-subspaces of the space of $q$-linearized polynomials over $\F_{q^n}$ of $q$-degree smaller than $n$, that is
$$ \mathcal L_{q,n}[x] \coloneqq \left\{ a_0x+ a_1x^q+\ldots+ a_{n-1}x^{q^{n-1}} \St a_i \in \F_{q^n} \right\}. $$
From any code $\tilde{\C}=\langle f_1(x), \ldots, f_k(x)\rangle_{\F_{q^n}}$ in this representation, then one can get a vector representation $\C$ as $[n,k]_{q^n/q}$ code by choosing an $\Fq$-basis $\mB=(\beta_1,\ldots, \beta_n)$ of $\F_{q^n}$ and defining 
$$\C:= \left\{ (f(\beta_1),\ldots, f(\beta_n)) \St f \in \tilde {\C} \right\}\subseteq \F_{q^n}^n.$$
However, independently of the choice of the $\Fq$-basis $\mB$, an $[n,k]_{q^n/q}$ system $\mU\in \Phi([\C])$ is given by
$$\mU =\left\{(f_1(\alpha),f_2(\alpha),\ldots, f_k(\alpha)) : \alpha \in \F_{q^n}\right\}. $$
Note that this special $q$-system was already studied in connection with the code $\tilde{\C}$ by Sheekey and Van de Voorde in \cite{sheekey2020rank}, where they used the notation $\mU_{\tilde{\C}}$.
\end{remark}

In order to illustrate the usefulness of the interplay between rank-metric codes and $q$-systems, we now give a notion of duality of $q$-systems using a coding theoretic perspective. Notice that this may be seen as a reinterpretation of the notion of Delsarte dual  of an $\Fmk$ system introduced in \cite{csajbok2021generalising}.

\begin{definition}\label{def:Delsartedual}
Let $1\leq k \leq n-1$ and let $\mU$ be an $\Fmkd$ system with $d\geq 2$. We define the \textbf{rank-metric dual} of $\mU$ to be any $[n,n-k]_{q^m/q}$ system in $\Phi([\C^\perp]),$
where $\C$ is any $\Fmkd$ code in $\Psi([\mU])$. We will denote it by $\mU^\perp$.
\end{definition}

\begin{remark}
Observe that $\mU^{\perp}$ is well-defined up to equivalence of $q$-systems, since it does not depend on the choice of the code $\C\in \Psi([\mU])$. This is because two codes $\C_1,\,  \C_2$ are equivalent if and only if $\C_1^\perp$ and $\C_2^\perp$ are equivalent.
Furthermore, the assumption on $d$ being at least $2$ implies, by Proposition \ref{prop:nondegenerate} that $\C^\perp$ is nondegenerate, and hence $\Phi([\C^\perp])$ is well-defined. In other words, we have the following diagram:

$$  \begin{tikzcd}[row sep=3em,column sep=5em]
\bigcup\limits_{d\geq 2} \mathfrak C(n,k,d)_{q^m/q}  \arrow[r,shift right,swap,"\Phi"]  \arrow[d,shift right,swap,"\perp"] & \bigcup\limits_{d\geq 2} \mathfrak U(n,k,d)_{q^m/q} \arrow[l,shift right,swap,"\Psi"] \arrow[d,shift right,swap,"\perp"]  \\%
\bigcup\limits_{d\geq 2} \mathfrak C(n,n-k,d)_{q^m/q} \arrow[u,shift right,swap,"\perp"] \arrow[r,shift right,swap,"\Phi"]  & \bigcup\limits_{d\geq 2} \mathfrak U(n,n-k,d)_{q^m/q} \arrow[l,shift right,swap,"\Psi"] \arrow[u,shift right,swap,"\perp"]
\end{tikzcd}
$$
\end{remark}

\begin{remark}
Observe that, in the case $n=m$, Definition \ref{def:Delsartedual} coincides with the \textbf{Delsarte dual} defined in \cite{csajbok2021generalising}. Indeed, although only stated for an MRD code, the result in \cite[Theorem 4.12]{csajbok2021generalising} can be extended to the more relaxed hypothesis that the code $\C$ and its dual $\C^\perp$ are both nondegenerate, Furthermore, we need to point out that the duality considered in that result is defined on the space $\mathcal L_{q,n}[x]$ and it  coincides with the one of Definition \ref{def:dualcode} when transforming codes in  $\mathcal L_{q,n}[x]$ to codes in $\Fm^m$ 
using an $\Fq$-basis $\mathcal B=(\beta_1,\ldots,\beta_m)$ of $\Fm$ which is self-dual with respect to the trace bilinear form, that is, such that
$$ \Tr_{\Fm/\Fq}(\beta_i,\beta_j)=\begin{cases} 1 & \mbox{ if } i=j \\
0 & \mbox{ if } i \neq j.\end{cases}$$
See \cite[Theorem 5.5]{neri2021twisted} for more details. 
\end{remark}

\medskip 
We conclude this section by recalling how the geometric point of view  allows us to characterize the generalized rank weights of an $\Fmk$  code via its associated $q$-system.

\begin{theorem}[\textnormal{\cite{randrianarisoa2020geometric}}]\label{thm:gen_weights_geometric}
 Let $\mC$ be a nondegenerate $\Fmk$  code and let $\mU\in\Phi([\C])$ be any of the $\Fmk$ systems associated. Then
 $$ \dd_{\rk,r}(\C)=n-\max\{\dim_{\Fq}(\mU\cap W) \St W\subseteq \Fm^k,  \dim_{\Fm}(W)=k-r\}.$$
\end{theorem}

We conclude this section with an illustrative example.

\begin{example}\label{exa:detailed}
 Let $q=2$, $m=3$ $k=2$ and $n=4$, and let us consider the field $\F_{2^3}=\F_2(\alpha)$, where $\alpha^3+\alpha+1=0$. Let $\C$ be the $[4,2,d]_{8/2}$ code whose generator matrix is 
 $$ G=\begin{pmatrix} 1 & 0 & \alpha & 0 \\
 0 & 1  & 0 &\alpha^2\end{pmatrix}.$$
 It is immediate to see that the columns of $G$ are $\F_2$-linearly independent, and hence the code $\C$ is nondegenerate.  One can compute a representative in $\Phi([\C])$ by taking the $\F_2$-subspace generated by the columns of $G$, obtaining the $[4,2,d]_{8/2}$ system
 $$\mU= \left\{(\lambda_1+\lambda_3\alpha, \lambda_2+\lambda_4\alpha^2) \St \lambda_1,\lambda_2,\lambda_3,\lambda_4 \in \F_2 \right\}\subseteq \F_8^2.$$
 The minimum distance $d$ of $\C$ can be  obtained by direct computation, observing that any scalar multiple of any of the two rows has rank weight $2$, and any nonzero linear combination of them can only increase the rank weight. On the other hand, it is also possible to find the value $d$ from $\mU$. The hyperplanes in $\F_8^2$ are one-dimensional $\F_8$-subspaces and 
 $$ d=4-\max \left\{ \dim_{\F_2}(\langle v\rangle_{\F_8} \cap \mU) \St v \in \F_8^2 \setminus \{0\} \right\}.$$
 We can immediately see that taking the hyperplane $H=\langle (0,1)\rangle_{\F_8}$, we have $\dim_{\F_2}(\mU\cap H)=2$, giving $d\leq 2$. However, $d=1$ if and only if there exists a vector $v\in\F_8^2$ such that $\langle  v\rangle_{\F_8}\subseteq \mU$. This is not possible, since the projections of $\mU$ on the first and on the second components are not surjective. Thus, $d=2$.
 
 Straightforward computations show that $\C^\perp$ is generated by the matrix 
 $$H=\begin{pmatrix} \alpha & 0 & 1 & 0 \\
 0 & \alpha^2 & 0 & 1
 \end{pmatrix}$$
 This implies that a representative for $\mU^\perp$ is the $\F_2$-subspace generated by the columns of $H$, that is,
 $$\mU^\perp= \left\{(\lambda_1\alpha+\lambda_3, \lambda_2\alpha^2+\lambda_4) \St \lambda_1,\lambda_2,\lambda_3,\lambda_4 \in \F_2 \right\}=\mU.$$ 
 Hence, we have that $\C^\perp$ and $\C$ are both $[4,2,2]_{8/2}$ nondegenerate codes. This also implies that $\dd_{\rk,2}(\C)=\dd_{\rk,2}(\C^\perp)=4$, and we  can see that the properties   of Proposition \ref{prop:gen_weights_properties} are satisfied: we have monotonicity and 
 $$ \{ \dd_{\rk,1}(\C), \dd_{\rk,2}(\C)\}=\{2,4\}, \qquad \{5-\dd_{\rk,2}(\C^\perp), 5-\dd_{\rk,1}(\C^\perp) \}=\{1,3\}. $$
 
\end{example}


\section{Evasive subspaces and generalized weights}\label{sec:evasive_genweights}

Evasive subspaces have been recently studied in \cite{bartoli2021evasive} as a special class of \emph{evasive sets}. The latter have been  introduced in \cite{pudlak2004pseudorandom} in connection with pseudorandomness and for constructing Ramsey graphs. Apart from being a special class of evasive sets, evasive subspace can also be seen as their $q$-analogues. Moreover, they are the natural generalization of $h$-scattered subapces, which were introduced first in \cite{blokhuis2000scattered} for $h=1$, and then studied in \cite{csajbok2021generalising} for arbitrary $h\in \mathbb{N}$.

Some first connections between evasive subspaces and rank-metric codes have been already observed in \cite{bartoli2021evasive}, in particular for what concerns the special class of $h$-scattered subspaces. This section is dedicated to Theorem \ref{thm:evasive_generalizedweights}, a  fundamental result which represents the foundations for the interplay between evasive subspaces and rank-metric codes.

\begin{definition}
Let $h,r,k,n$ be nonnegative integers such that $h<k \leq n$. An $\Fmk$ system $\mU$ is said to be  \textbf{$(h,r)$-evasive} if for every $h$-dimensional $\Fm$-subspace $W$ of $\Fm^k$ we have $\dim_{\Fq}(\mU\cap W)\leq r$. When $h=r$, an $(h,h)$-evasive $\Fmk$ system will be also called \textbf{$h$-scattered}. If moreover $h=1$, we will simply say that it is \textbf{scattered}.
\end{definition}

\begin{remark} Here some properties of evasive $q$-systems.
 \begin{enumerate}
     \item If $r<h$ there is no $(h,r)$-evasive $\Fmk$ system.
     \item If $\mU$ is an $(h,r)$-evasive $\Fmk$ system and $h>0$, then $\mU$ is also $(h-1,r-1)$-evasive.
 \end{enumerate}
\end{remark}

The following result gives a precise description of the connections between evasive subspaces and rank-metric codes.

\begin{theorem}\label{thm:evasive_generalizedweights}
Let $\mC$ be a nondegenerate  $\Fmk$ code, and let $\mU\in\Phi([\mC])$. Then, the following are equivalent
\begin{enumerate}
 \item $\mU$ is  $(h,r)$-evasive.
 \item $\dd_{\rk,k-h}(\C) \geq n-r$.
 \item $\dd_{\rk,r-h+1}(\C^\perp)\geq r+2$.
\end{enumerate}

  In particular, $\dd_{\rk,k-h}(\C) = n-r$  if and only if $\mU$ is $(h,r)$-evasive but not $(h,r-1)$-evasive if and only if  $\dd_{\rk,r-h+1}(\C^\perp)\geq r+2 \geq \dd_{\rk,r-h}(\C^\perp)+2$.
\end{theorem}

\begin{proof}
 The equivalence between (1) and (2) follows from the definition of evasive subspace and from Theorem \ref{thm:gen_weights_geometric}. 
 Now, by Proposition \ref{prop:gen_weights_properties}(1), we have $\dd_{\rk,k-h}(\C) \geq n-r$ if and only if $\{\dd_{\rk,k-h+i}(\C) \St 0\leq i \leq h\} \subseteq \{n-r,\ldots,n\}$, or in other words
 $$|\{\dd_{\rk,i}(\C) \St 1\leq i \leq k\}\cap\{n-r,\ldots,n\}|\geq h+1.$$  
 By Proposition \ref{prop:gen_weights_properties}(2), this is true if and only if
 $$|\{n+1-\dd_{\rk,i}(\C^\perp) \St 1\leq i \leq n-k\}\cap\{n-r,\ldots,n\}|\leq r-h.$$
 Again using Proposition \ref{prop:gen_weights_properties}(2), this is equivalent to say that 
 $$|\{n+1-\dd_{\rk,i}(\C^\perp) \St 1\leq i \leq n-k\}\cap\{1,\ldots,n-r-1\}|\geq n-k-r+h.$$
 In other words,
 $\{n+1-d_{\rk,i}(\C^\perp) \St r-h+1 \leq i \leq n-k\} \subseteq \{1,\ldots,n-r-1\},$
 which is in turn equivalent to say (by Proposition \ref{prop:gen_weights_properties}(1)) that $n+1-d_{\rk,r-h+1}(\C^\perp)\le n-r-1$, that is condition (3).
 
 Finally, the equivalence between the fact that $\dd_{\rk,k-h}(\C) = n-r$ and that $\mU$ is  $(h,r)$-evasive but not $(h,r-1)$-evasive immediately follows from the first part of the statement. In the same way, we can obtain that these are in turn equivalent to say that $\dd_{\rk,r-h+1}(\C^\perp)\geq r+2$ and $\dd_{\rk,r-h}(\C^\perp)\leq r$.
\end{proof}

We conclude also this section with an illustrative example.

\begin{example}\label{exa:duality}
 Let $q=2$, $m=3$, $k=2$ and $n=5$, and let us consider the field $\F_{2^3}=\F_2(\alpha)$, where $\alpha^3+\alpha+1=0$. Let $\C$ be the $[5,2,d]_{8/2}$ code whose generator matrix is 
 $$ G=\begin{pmatrix} 1 & 0 & \alpha & \alpha^2 & 1\\
 0 & 1  & 1 & 0 &\alpha\end{pmatrix}.$$
 The code $\C$ is nondegnerate, since the columns of $G$ are $\F_2$-linearly independent, and a $[5,2,d]_{8/2}$ system associated to $\C$ is given by the $\F_2$-subspace generated by the columns, that is
 \begin{align*}\mU&=\left\{(\lambda_1 + \lambda_3\alpha + \lambda_4 \alpha^2 +\lambda_5, \lambda_2+\lambda_3+\lambda_5\alpha ) \St \lambda_1, \lambda_2, \lambda_3,\lambda_4,\lambda_5 \in \F_2 \right\}\\
 &= \left\{(\lambda_1 + \lambda_3\alpha + \lambda_4 \alpha^2, \lambda_2+\lambda_5\alpha ) \St \lambda_1, \lambda_2, \lambda_3,\lambda_4,\lambda_5 \in \F_2 \right\} \subseteq \F_8^2.
 \end{align*}
 It is immediate to see that $\mU$ is $(1,3)$-evasive, and hence $\dd_{\rk}(\C)=2$. Moreover, being $\C$ nondegenerate, we also have $\dd_{\rk,2}(\C)=5$.

 Furthermore, from the change of basis of $\mU$, we get that we can consider an equivalent code $\C'$ to $\C$ whose generator matrix is obtained as
 $$ G'=\begin{pmatrix}  1 & 0 & \alpha & \alpha^2 & 0\\
 0 & 1  & 0 & 0 &\alpha
 \end{pmatrix}.$$
 Now, let us consider the dual code $\C'^\perp$ which is equivalent to $\C^\perp$. A generator matrix for $\C^\perp$ is given by
 $$ H'=\begin{pmatrix} \alpha & 0 & 1 & 0 & 0\\
 \alpha^2 & 0  & 0 & 1 & 0  \\ 
 0 & \alpha & 0 & 0 & 1\end{pmatrix}$$
 Thus, a representative $\mU^\perp$ of the rank-metric dual of $\mU$ is given by the $\F_2$-subspace of the columns of $H$, that is,
 $$ \mU^\perp=\left\{ (\lambda_1\alpha+\lambda_3, \lambda_1\alpha^2 +\lambda_4, \lambda_2\alpha+\lambda_5) \St \lambda_1, \lambda_2, \lambda_3,\lambda_4,\lambda_5 \in \F_2 \right\}.$$
 Using the Wei-type duality in Proposition \ref{prop:gen_weights_properties}(2), we deduce that 
 $$ \dd_{\rk}(\C^\perp)=2, \qquad  \dd_{\rk,2}(\C^\perp)=3, \qquad \dd_{\rk,3}(\C^\perp)=5.$$
 Finally, by applying Theorem \ref{thm:evasive_generalizedweights} to the code $\C^\perp$, we deduce that 
 $\mU^\perp$ is $(1,2)$-evasive and $(2,3)$-evasive.
\end{example}


\section{Consequences of Theorem \ref{thm:evasive_generalizedweights}} \label{sec:consequences}

Theorem \ref{thm:evasive_generalizedweights} has multiple consequences, which we  carefully describe in this section. Several of these results are reinterpretations of recent findings, while some  are new. However, our approach unifies the  techniques used in a single framework and simplifies significantly all the proofs previously proposed.

\subsection{Connections between codes and evasive subspaces}
The first implication of Theorem \ref{thm:evasive_generalizedweights} is that it allows to characterize codes associated to $h$-scattered subspaces, which turn out to be $(k-h)$-MRD.  

\begin{corollary}\label{cor:hscattered_sMRD}
 Let $\C$ be a nondegenerate $\Fmk$  code and let $\mU\in\Phi([\C])$. Then $\mU$ is an $h$-scattered $\Fmk$ system if and only if $\C$ is $(k-h)$-MRD.
\end{corollary}

Theorem \ref{thm:evasive_generalizedweights} also implies the existence of nontrivial $s$-MRD codes.

\begin{corollary}[Existence of $s$-MRD codes]\label{cor:sMRD_exist} Assume that $km$ is even. Then, for every $k\leq n \leq \frac{km}{2}$ there exists nondegenerate $\Fmk$ codes that are $(k-1)$-MRD.  
\end{corollary}

\begin{proof}
 The existence of a scattered $[\frac{km}{2},k]_{q^m/q}$ system $\mU$  has been proved in \cite{csajbok2017maximum} (see also Theorem \ref{thm:existence_maxScattered}). Clearly, any $n$-dimensional $\Fq$-subspace of $\mU$ is still scattered. The existence of an $\Fmk$ $(k-1)$-MRD then follows from Corollary \ref{cor:hscattered_sMRD}.
\end{proof}

\begin{corollary}
Assume  $0<k<n$, let $\mU$ be a nondegenerate $\Fmkd$ system with $d\geq 2$ and let $\mU^{\perp}$ be the rank-metric dual, which is an $[n,n-k]_{q^m/q}$ system. Then 
$\mU$ is $(h,r)$-evasive if and only if $\mU^{\perp}$ is $(n-k-r+h-1,n-r-2)$ evasive.
\end{corollary}

\begin{proof}
  Since $d\geq 2$, we have that  $\mU^{\perp}$ is well-defined and it is an $[n,n-k]_{q^m/q}$ system. The parameters  can be derived from the equivalent properties in Theorem \ref{thm:evasive_generalizedweights}
\end{proof}

\subsection{Bounds on evasive subspaces}

Other important results can be derived from Theorem \ref{thm:evasive_generalizedweights}. Here, we focus on inequalities among their parameters. The first one, is the known upper bound on the maximum rank of an $h$-scattered subspace. Our  proof is immediate and only consists of few lines. 

\begin{corollary}[$h$-Scattered Bound \textnormal{\cite[Theorem 2.3]{csajbok2021generalising}}]\label{cor:hscattered_bound}
Let $\mU$ be an $h$-scattered $\Fmk$ system. Then
$$ n\leq \frac{km}{h+1}.$$
\end{corollary}

\begin{proof}
 Let $\C\in \Psi([\mU])$ be any (nondegenerate) $\Fmk$ code associated to $\mU$. By Theorem \ref{thm:evasive_generalizedweights}, we have $\dd_{\rk}(\C^\perp)=\dd_{\rk,1}(\mC^\perp)\geq h+2$. On the other hand, by  \eqref{eq:singleton}, we also have 
 $$m(n-k)\leq n(m-\dd(\C^\perp)+1).$$
 Combining the two inequalities, we obtain the desired result.
\end{proof}

We now derive bounds on the dimension  for the special case of $(k-1,r)$-evasive subspaces. The following result is essentially \cite[Theorem 4.2]{bartoli2021evasive}.

\begin{corollary}\label{cor:k-1evasive}
 Let $\mU$ be a $(k-1,r)$-evasive $\Fmk$ system. Then
 $$ km \leq n(m-n+r+1).$$
\end{corollary}

\begin{proof}
 It immediately follows combining Theorem \ref{thm:evasive_generalizedweights} with the Singleton bound in \eqref{eq:singleton}.
\end{proof}

Furthermore, we can also deduce bounds on general $(h,r)$-evasive subspaces.

\begin{corollary}[Evasive Bounds]\label{cor:evasive_bounds}
Let $\mU$ be an $(h,r)$-evasive $\Fmk$ system. The following hold.
\begin{enumerate}\setlength\itemsep{1em}
    \item $r \geq \max\{h,h-1+\frac{h+1}{m-1}\}.$
    \item  $n\leq km-hm+r$ (see \cite[Theorem 4.3]{bartoli2021evasive}). 
    \item If $r\leq \frac{m}{m-1}h$, then
$$ n\leq \frac{km}{r+1-m(r-h)}.$$
\end{enumerate}

\end{corollary}

\begin{proof}
 Let $\mU$ be an $(h,r)$-evasive subspace, and let $\C \in \Psi([\mU])$ be any (nondegenerate) $\Fmk$ code associated. 
 \begin{enumerate}
\item By Theorem \ref{thm:evasive_generalizedweights}, we have $\dd_{\rk,r-h+1}(\mC^\perp)\geq r+2$. On the other hand, by  Proposition \ref{prop:bounds_genweights}, we also have $\dd_{\rk,r-h+1}(\mC^\perp)\le (r-h+1)m$. Combining the two bounds we obtain the desired result.
\item By Theorem \ref{thm:evasive_generalizedweights}, we have $\dd_{\rk,k-h}(\mC)\geq n-r$. On the other hand, by Propostion \ref{prop:bounds_genweights}, we also have 
 $\dd_{\rk,k-h}(\C^\perp)\leq(k-h)m$, and combining the inequalities we obtain the desired bound.
\item By Theorem \ref{thm:evasive_generalizedweights}, we have $\dd_{\rk,r-h+1}(\mC^\perp)\geq r+2$. On the other hand, by  Proposition \ref{prop:bounds_genweights}, we also have 
 $$\dd_{\rk,r-h+1}(\C^\perp)\leq\frac{m}{n}k+m(r-h)+1 $$
 Combining the two inequalities, we obtain the desired result.
 \end{enumerate}
\end{proof}

We conclude with the following nonexistence result for $h$-scattered subspaces.

\begin{corollary}\label{cor:existence_hscattered}
 If $m<h+2$, there are no $h$-scattered $\Fmk$ systems.
\end{corollary}

\begin{proof}
 It follows directly from Corollary \ref{cor:evasive_bounds}(1).
\end{proof}

\subsection{Maximum $h$-scattered subspaces and MRD codes}
In this section we derive already known connections between MRD codes and maximum $h$-scattered subspaces of $\Fm^k$. However, although the results are already known, using our approach we can derive them in a very simple and elementary way, showing how powerful is the result of Theorem \ref{thm:evasive_generalizedweights}.

An $\Fmk$ system is said to be \textbf{maximum $h$-scattered} if it is $h$-scattered and meets the bound of Corollary \ref{cor:hscattered_bound} with equality, that is,
$$ n=\frac{km}{h+1}.$$
We point out that maximum $h$-scattered $\Fmk$ systems do exist.

\begin{theorem}\label{thm:existence_maxScattered} Let $k,m$ be positive integers and let $q$ be a prime power. The following hold.
\begin{enumerate}
 \item There exist maximum $(k-1)$-scattered $[m,k,m-k+1]_{q^m/q}$ systems (see \cite{delsarte1978bilinear}).
  \item Assume that $km$ is even. Then, there exist maximum scattered $[\frac{km}{2},k,]_{q^m/q}$ systems (see \cite{blokhuis2000scattered,bartoli2018maximum} for the first partial results and \cite{csajbok2017maximum} for the complete proof).
  \end{enumerate}
\end{theorem}
It is important to remark the fact that both the existence results are equipped with a constructive proof. Part (1) is the case of the well-known Gabidulin codes (see also \cite{gabidulin1985theory}), even though other constructions are known \cite{sheekey2016new}. Part (2) was first proved by Blokhuis and Lavrauw in \cite{blokhuis2000scattered} when $k$ is even. The case $k$ odd and $m$ even was first addressed in \cite{bartoli2018maximum} for almost all the parameters, and then completed in \cite{csajbok2017maximum}.

In light of Theorem \ref{thm:evasive_generalizedweights}, we can deduce a further connection between the largest $h$-scattered subspaces and special rank-metric codes.

\begin{theorem}[\textnormal{\cite[Theorem 3.2]{zini2021scattered}}]\label{thm:hscattered_MRD}
 Suppose that $h+1$ divides $km$ and let $n:=\frac{km}{h+1}$. Let $\mU$ be an $\Fmk$ system and let $\mC\in \Psi([\mU])$ be any of its associated $\Fmk$ codes. Then, $\mU$ is maximum $h$-scattered if and only if $\mC$ is an MRD code.
\end{theorem}

\begin{proof}
  By Theorem \ref{thm:evasive_generalizedweights}, we have that ${\mU}$ is $h$-scattered if and only if $\dd_{\rk}(\C^{\perp})\geq h+2$. Moreover, the $\Fq$-dimension of ${\mU}$ is $n=\frac{km}{h+1}$ if and only if $\C^{\perp}$ is a $[\frac{km}{h+1},\frac{km}{h+1}-k]_{q^m/q}$ code. On the other hand, by \eqref{eq:singleton}, any $[\frac{km}{h+1},\frac{km}{h+1}-k]_{q^m/q}$ has minimum distance at most $h+2$, and it has minimum distance exactly $h+2$ if and only if \eqref{eq:singleton} is met with equality. By \cite{delsarte1978bilinear,gabidulin1985theory}, $\C^\perp$ is MRD if and only if $\C$ is MRD, and this concludes the proof.
\end{proof}

From this, we immediately derive a very short proof of the following result.

\begin{corollary}[\textnormal{\cite[Theorem 2.7]{csajbok2021generalising}}]\label{cor:maxscatterd_hyperplane_intersection}
 Let ${\mU}$ be a maximum $h$-scattered $\Fmk$ system. Then, for any $\Fm$-hyperplane $H \subseteq \Fm^k$, we have
 $$ \frac{km}{h+1}-m \leq \dim_{\Fq}(\mU\cap H) \leq \frac{km}{h+1}-m+h.$$
\end{corollary}

\begin{proof}
 By Theorem \ref{thm:hscattered_MRD}, ${\mU}$ is maximum $h$-scattered if and only if any $[\frac{km}{h+1},k]_{q^m/q}$ code $\C\in\Psi([\mU])$ is MRD. Now, every codeword of $\C$ has rank at most $m$, and at least its minimum distance, which by \eqref{eq:singleton} is $m-h$. Now, by the correspondence of Theorem \ref{thm:correspondence_codes_systems} (see also \eqref{eq:rank_weight}), this means that every $\Fm$-hyperplane $H \subseteq \Fm^k$ is such that 
 $$ \frac{km}{h+1}-m \leq \dim_{\Fq}(\mU\cap H) \leq \frac{km}{h+1}-m+h.$$
\end{proof}

\begin{remark}\label{rem:maxhscattered_evasive}
From the result of Corollary \ref{cor:maxscatterd_hyperplane_intersection}, we can immediately deduce that any  maximum $h$-scattered  $[\frac{km}{h+1},k]_{q^m/q}$ system is also $(k-1,\frac{km}{h+1}-m+h)$-evasive.
\end{remark}

\section{Near MRD codes}\label{sec:nearMRD}

As for codes in the Hamming metric, we can study codes in the rank metric which are as close as possible to be $1$-MRD. This could result in a gain on the length of a code, without loosing too much on the distance properties. We have already seen in Corollary \ref{cor:sMRD_exist}  that  $(k-1)$-MRD codes can be much longer than $1$-MRD codes.
Here, we introduce the notion of near MRD codes, in analogy with their counterparts in the Hamming metric given by near MDS codes. These codes have already been studied in \cite{de2018dually} under the name of $2$-AMRD codes. However, in order to be more consistent with the literature in Hamming-metric coding theory, we prefer to call them \emph{near MRD codes}. 

\subsection{Geometric characterization and bounds}

Here we give a definition of near MRD codes in analogous way to what is done for near MDS codes in the Hamming metric. With the aid of Theorem \ref{thm:evasive_generalizedweights}, we then analyze geometrically what these codes correspond to. This geometric point of view will help us in deriving very interesting results on the parameters of near MRD codes. 

\begin{definition}
 An $\Fmk$ code $\C$ is called \textbf{near MRD} if  $\dd_{\rk}(\C)=n-k$ and $\dd_{\rk,s}(\C)=n-k+s$ for every $2\leq s\leq k$.
\end{definition}

Observe that, due to the monotonicity of the generalized rank weights (Proposition  \ref{prop:gen_weights_properties}(2)) and their bounds  (Proposition \ref{prop:bounds_genweights}), it is enough to consider just the first two generalized weights. In other words, an $\Fmk$ code $\C$ is near MRD if and only if $\dd_{\rk}(\C)=n-k$ and $\dd_{\rk,2}(\C)=n-k+2$.

\begin{example}
Let $q=2$ and $m=3$, and consider $\F_{8}=\F_2(\alpha)$, where $\alpha^3+\alpha+1=0$. Consider the $[4,2]_{2^3/2}$ code $\C$ studied in Example  \ref{exa:detailed}, whose generator matrix is
$$ G=\begin{pmatrix} 1 & 0 & \alpha & 0 \\ 0 & 1 & 0 & \alpha^2
\end{pmatrix}.$$
We have seen before that $\dd_{\rk}(\C)=2$ and  $\dd_{\rk,2}(\C)=4$. Therefore, $\C$ is a near MRD code.  
\end{example}

We now give a geometric characterization for near MRD codes. 

\begin{proposition}\label{prop:nearMRD}
 Let $\C$ be a nondegenerate $\Fmk$  code and let $\mU\in\Phi([\C])$. The following are equivalent.
 \begin{enumerate}
     \item $\C$ is near MRD.
     \item $\dd_{\rk}(\C)+\dd_{\rk}(\C^\perp)=n$, that is, $\C$ is \textit{dually almost MRD}.\footnote{The notion of dually almost MRD codes was introduced and studied in \cite{de2018dually}.}
     \item \begin{enumerate}
         \item $\mU$ is $(k-2)$-scattered,
         \item $\mU$ is not $(k-1)$-scattered, and
         \item $\mU$ is $(k-1,k)$-evasive.
     \end{enumerate}
 \end{enumerate}
\end{proposition}

\begin{proof}
 The equivalence between (1) and (2) derives immediately from Proposition \ref{prop:gen_weights_properties}; see also \cite[Proposition 3.4]{de2018dually}. The equivalence between (1) and (3) follows instead from Theorem \ref{thm:evasive_generalizedweights}.
\end{proof}

In the case one aims to construct codes of length $n\leq m$, then it is well-known that one can easily construct MRD codes. Hence, it is certainly of more significance to construct longer codes, whose parameters are not excluded by Proposition \ref{prop:BoundNearMRD}. In such a case, one can just check only two out of the three conditions on the associated $\Fmk$ system $\mU\in\Phi([\C])$ of Proposition \ref{prop:nearMRD}. Indeed, by \eqref{eq:singleton}, an $[m+i,k]_{q^m/q}$ system with $i \ge 1$ cannot be $(k-1)$-scattered, and hence one only needs to verify the remaining two conditions. Using these conditions, we get the following result.

\begin{proposition}\label{prop:BoundNearMRD}
Let $\C$ be a nondegenerate $\Fmk$ near MRD code. We have
$$n\leq \min \bigg\{\frac{km}{k-1}, m+t \bigg\},$$
where $t=\max\{ i \St i^2+(m-k-1)i-m \leq 0\}$. 
In other words,
$$n \leq \min \bigg\{\frac{km}{k-1},m+\frac{k-m+1+\sqrt{m^2-2(k-1)m+(k+1)^2}}{2}\bigg\}.$$
\end{proposition}

\begin{proof}
 Use Corollary \ref{cor:hscattered_bound} and Corollary \ref{cor:k-1evasive}.
\end{proof}

We are now ready to show the main result of this section, which provides clean bounds on the parameters $n,m,k$  of a $[n,k,n-k]_{q^m/q}$ near MRD code, and geometrically characterize the longest near MRD codes. 

\begin{theorem}\label{thm:nearMRD_characterization}
 Let $\C$ be an $\Fmk$ near MRD code. Then $m \geq k$ and 
 $$ n \leq \begin{cases} m+2 & \mbox{ if } m=2k-2, \\
 m+1 & \mbox{ if }  m \neq 2k-2.
 \end{cases}$$
 Moreover, $\C$ is a $[2k,k,k]_{q^{2k-2}/q}$ near MRD code if and only if any $\mU \in \Phi([\C])$ is a maximum $(k-2)$-scattered subspace. 
\end{theorem}

\begin{proof}
  First, using Corollary \ref{cor:existence_hscattered} there is no $(k-2)$-scattered $\Fmk$ system if $m<k$. Thus, by Proposition \ref{prop:nearMRD}, we must have $m \geq k$. 
  
  Concerning the upper bound on $n$, we divide the proof in three cases:
  
  \noindent When $m=2k-2$, then, using Proposition \ref{prop:BoundNearMRD}, we obtain $n \leq 2k=m+2$.
  
  \noindent   Assume now that $m>2k-2$. Then one can easily  verify that 
   $\max\{i \St i^2+(m-k-1)i-m \leq 0\}=1$. 
  Also in this case, using Proposition \ref{prop:BoundNearMRD}, we deduce $n \leq m+1$.
  
  \noindent   If we suppose instead that $m<2k-2$, we may write $m=2k-2-\epsilon$ for some $\epsilon\ge 1$, and obtain that $\frac{km}{k-1}=2k-\epsilon\frac{k}{k-1}< 2k-\epsilon=m+2$. Thus, $n \le m+1$ and we conclude using again Proposition \ref{prop:BoundNearMRD}.
  
  For the last part of the statement, if $\C$ is a $[2k,k,k]_{q^{2k-2}/q}$ near MRD code then by Proposition \ref{prop:nearMRD} any $\mU\in \Phi([\C])$ is maximum $(k-2)$-scattered. Viceversa, if $\mU\in \Phi([\C])$ is a maximum $(k-2)$-scattered $[2k,k,k]_{q^{2k-2}/q}$ system, then by Corollary \ref{cor:maxscatterd_hyperplane_intersection} (see also Remark \ref{rem:maxhscattered_evasive}), we have that $\mU$ is also $(k-1,k)$-evasive. Moreover, by \eqref{eq:singleton}, $\mU$ cannot be $(k-1)$-scattered, and using Proposition \ref{prop:nearMRD} we derive that $\C$ is a $[2k,k,k]_{q^{2k-2}/q}$ near MRD code.
\end{proof}

We now give a first nontrivial construction/example of near MRD codes, that is, near MRD codes  whose length is larger than $m$. 

\begin{proposition}\label{prop:nearMRD_firstConstruction} Assume that $m \geq k$. Then, 
the set $$\mU\coloneqq\{(\alpha+\lambda, \alpha^q, \ldots,\alpha^{q^{k-1}})\St \alpha \in \Fm, \lambda \in \Fq\}$$ is an $[m+1,k]_{q^m/q}$ system, which is $(k-2)$-scattered and $(k-1,k)$-evasive.
\end{proposition}

\begin{proof}
Suppose on the contrary that there exists a $(k-1)$-dimensional $\Fq$-subspace $\mV$ of $\mU$ whose $\Fm$-span has $\Fm$-dimension at most $k-2$.  An $\Fq$-basis of $\mV$ can be taken to be of the form $$\{(\alpha_1+\lambda_1,\alpha_1^q,\ldots,\alpha_1^{q^{k-1}})\}\cup\{(\alpha_i,\alpha_i^q,\ldots,\alpha_i^{q^{k-1}}) \St 2\le i\le k-1)\}, $$
for some $\alpha_1,\ldots \alpha_{k-1}\in \Fm$,$\lambda_1\in\Fq$. First, observe that $\alpha_2,\ldots,\alpha_{k-1}$ must be $\Fq$-linearly independent, otherwise the last $k-2$ vectors would be $\Fq$-linearly dependent and hence the set above would not be an $\Fq$-basis of $\mV$. If in addition the $\alpha_i$'s are all $\Fq$-linearly independent, then the $\Fm$-span of $\mV$ has dimension $k-1$. Thus, we must have $\alpha_1\in\langle \alpha_2,\ldots,\alpha_{k-1}\rangle_{\Fq}$. This means that we can get another basis of $\mV$ that is given by 
$$\{(\lambda_1,0,\ldots,0)\}\cup\{(\alpha_i,\alpha_i^q,\ldots,\alpha_i^{q^{k-1}}) \St 2\le i\le k-1)\}, $$
implying that $\lambda_1\neq 0$. Putting these vectors as columns of a $k\times (k-1)$ matrix, we have that the $(k-1)\times(k-1)$ minor obtained deleting the last row is $\lambda_1\gamma^q$, where $\gamma$ is the determinant of the $(k-2)\times (k-2)$ Moore matrix on the elements $\alpha_2,\ldots,\alpha_{k-1}$, that are $\Fq$-linearly independent.
Therefore, these vectors are also $\Fm$-linearly independent, getting a contradiction. Thus, $\mU$ is $(k-2)$-scattered.

Moreover, observe that $\mU$ contains the $[m,k]_{q^m/q}$ system $\{(\alpha, \alpha^q, \ldots,\alpha^{q^{k-1}})\St \alpha \in \Fm\}$, which is of pseudoregulus type and hence $(k-1)$-scattered. This implies that $\mU$ is $(k-1,k)$-evasive. 
\end{proof}

\begin{remark} Proposition \ref{prop:nearMRD_firstConstruction} automatically proves that any code  $\C\in\Psi([\mU])$ is an $[m+1,k]_{q^m/q}$ near MRD code; see Proposition \ref{prop:nearMRD}. Moreover, by Theorem \ref{thm:nearMRD_characterization}, this is the longest possible near MRD code for given $k,m$ (and $q$), whenever $m\neq 2k-2$.
\end{remark}

We show with a concrete example a near MRD code obtained from Proposition \ref{prop:nearMRD_firstConstruction}.

\begin{example}
 Let $q=2, m=4$ and $k=3$. Consider the finite field $\F_{16}=\F_2(\beta)$, where $\beta^4+\beta+1=0$. Since $\{1,\beta,\beta^2,\beta^3\}$ is an $\F_2$-basis of $\F_{16}$, we have that  
 $$\{(1,0,0),(1,1,1),(\beta,\beta^2,\beta^4),(\beta^2,\beta^4,\beta^8),(\beta^3,\beta^{6},\beta^{12}) \}$$ is an $\F_2$-basis for $\mU$. Thus, a $[5,3]_{2^4/2}$ near MRD code $\C \in \Psi([\mU])$ is the one generated by the matrix 
 $$ G=\begin{pmatrix} 1 & 1 & \beta & \beta^2  & \beta^3 \\
 0 & 1 & \beta^2 & \beta^4 & \beta^6  \\
 0 & 1 & \beta^4 & \beta^8 & \beta^{12}  
 \end{pmatrix}.$$
\end{example}

\subsection{Three-dimensional and four-dimensional near MRD codes}\label{sec:3-4-dim}

We now analyze the special case of three-dimensional near MRD codes. The bound of Proposition \ref{prop:BoundNearMRD} simplifies as follows. Let $\C$ be an $[n,3]_{q^m/q}$ near MRD code. Then

\begin{equation*}
    n \leq \begin{cases} m+1 & \mbox{ if } m \geq 5, \mbox{ or } m =3, \\
    m+2 &  \mbox{ if } m=4.
    \end{cases}
\end{equation*}

Observe that for each $m \geq 3$, we already have an $[m+1,3]_{q^m/q}$ near MRD code, due to Proposition \ref{prop:nearMRD_firstConstruction}. Moreover, 
     any maximum scattered $[6,3,3]_{q^4/q}$ system $\mU$  gives a $[6,3,3]_{q^4/q}$ near MRD code $\C\in \Psi([\mU]))$; see Theorem \ref{thm:nearMRD_characterization}.  The existence of such subspaces has already been proved in \cite{csajbok2017maximum}, which includes a constructive proof; see Theorem \ref{thm:existence_maxScattered}(2). This completes the picture on the longest three-dimensional near MRD codes. 
     
    \begin{remark}
    We want to point out that  $[6,3,3]_{q^4/q}$ near MRD codes have already been encountered in \cite{alfarano2021linear}. Indeed, in that paper, the authors showed that whenever $\mU$ is a scattered $[m+2,3]_{q^m/q}$ system, then any $\C\in \Psi([\mU])$  provides an example of the shortest  $3$-dimensional \emph{minimal} rank-metric codes over $\F_{q^m}$.
    \end{remark}

We now analyze the special case of near MRD codes with dimension $4$. The bound of Proposition \ref{prop:BoundNearMRD} can be rewritten in the following way. Let $\C$ be an $[n,4]_{q^m/q}$ near MRD code. Then
\begin{equation}
    n \leq \begin{cases} m+1 & \mbox{ if } m \geq 7 \mbox{ or } 4\le m \le 5 \\
    m+2 &  \mbox{ if }  m = 6.
    \end{cases}
\end{equation}

Observe that for each $m \geq 4$ we already have an $[m+1,4]_{q^m/q}$ near MRD code, due to Proposition \ref{prop:nearMRD_firstConstruction}. The only case left is if we can construct an $[8,4,4]_{q^6/q}$ near MRD code. Note that, by Theorem \ref{thm:nearMRD_characterization}, this is equivalent to construct a $2$-scattered $[8,4]_{q^6/q}$ system. So far, there is only one known construction of $2$-scattered $[8,4]_{q^6/q}$ systems, which was shown in \cite{BMN2} for every $q=2^h$, with $h\equiv 1 \pmod 6$.

\subsection{Relation between the families of codes}

We conclude the section analyzing the relations between the families of MRD codes, $1$-MRD codes, $2$-MRD codes and near MRD codes. Let us introduce the following notation, defining the sets
$$ \begin{array}{rcl} 
 \A(q)& \coloneqq & \left\{ \C \St \C \mbox{ is an } \Fmk \mbox{ MRD code, for some } n,k,m \in \N \right\}, \\
 \A_{\mathrm{1}}(q)& \coloneqq & \left\{ \C \St \C \mbox{ is an } \Fmk \mbox{ $1$-MRD code, for some } n,k,m \in \N \right\}, \\
 \A_{\mathrm{2}}(q)& \coloneqq & \left\{ \C \St \C \mbox{ is an } \Fmk \mbox{ $2$-MRD code, for some } n,k,m \in \N \right\}, \\
 \A_{\mathrm{N}}(q)& \coloneqq & \left\{ \C \St \C \mbox{ is an } \Fmk \mbox{ near MRD code, for some } n,k,m \in \N \right\}.
\end{array} $$
 It is clear by the definition that 
  $$ \A_{\mathrm{N}}(q)\cup \A_{\mathrm{1}}(q) \subseteq \A_{\mathrm{2}}(q) \quad \mbox{ and } \quad \A_{\mathrm{N}}(q)\cap \A_{\mathrm{1}}(q)=\emptyset.$$
 Furthermore, as already mentioned in Remark \ref{rem:MRDvs1MRD}, $1$-MRD codes are also MRD, but the viceversa is not true. 
 
 \begin{proposition}
    There exist an MRD code $\C_1$ which is not $2$-MRD.. In other words,
  $$  \A_{\mathrm{2}}(q) \subsetneq \A(q). $$
 \end{proposition}
 
 \begin{proof}
   Let $mk$ be even with $k \geq 4$ and let $\mU$ be any maximum scattered $[\frac{km}{2},k]_{q^m/q}$ system. Then any code $\C_1\in\Psi([\mU])$ is MRD (Theorem \ref{thm:hscattered_MRD}). If $\C_1$ is  also $2$-MRD, then by Corollary \ref{cor:hscattered_sMRD} it is also $(k-2)$-scattered. However, this is not possible by Corollary \ref{cor:hscattered_bound}, since $k-2 \geq 2$.   
 \end{proof}
 
 \begin{proposition}
  The codes $\C_2$ constructed in Proposition \ref{prop:nearMRD_firstConstruction} are near MRD but not MRD. In particular, 
  $$ \A_{\mathrm{N}}(q)\setminus \A(q) \neq \emptyset..$$
 \end{proposition}
 
 \begin{proof}
   The codes $\C_2$ are $[m+1,k,m+1-k]_{q^m/q}$ near MRD codes. It is clear from their parameters that they cannot be MRD.
 \end{proof}
 
 \begin{proposition}
  For every prime power $q$, we have 
  $$ \A_{\mathrm{N}}(q)\cap \A(q) = \left\{ \C \St \C \mbox{ is a } [2k,k,k]_{q^{2k-2}/q} \mbox{ code, for some } k \in \N \right\}. $$
  Furthermore, there exists an MRD code $\C_3$ which is also near MRD. In other words,
  $$ \A_{\mathrm{N}}(q)\cap \A(q) \neq \emptyset.$$
 \end{proposition}

 \begin{proof}
By Theorem \ref{thm:nearMRD_characterization} we immediately deduce that  
 $$\A_{\mathrm{N}}(q)\cap \A(q) \supseteq  \left\{ \C \St \C \mbox{ is a } [2k,k,k]_{q^{2k-2}/q} \mbox{ code, for some } k \in \N \right\}.$$
  Let $\C$ be an $\Fmkd$ code in $\A_{\mathrm{N}}(q)\cap \A(q)$. By Theorem \ref{thm:nearMRD_characterization}, we have $n \leq m+2$. The case $n=m+2$ coincides with the one just described. Thus, we only need to prove that we cannot have $n\leq m+1$.  If $n\leq m$, then the code $\C$ would also be $1$-MRD, but it cannot be $1$-MRD and near MRD at the same time. Furthermore, if $n=m+1$, then $\C$ is an $[m+1,k,m+1-k]_{q^m/q}$ code, which is not MRD, since  $$ km< (m+1)k=\min\{m(n-d+1),n(m-d+1)\}=n(m-d+1).$$
   This shows the first claim. Finally, the $[6,3,3]_{q^4/q}$ codes $\C_3$ discussed in Section \ref{sec:3-4-dim} are near MRD (Theorem \ref{thm:nearMRD_characterization}) and MRD (Theorem \ref{thm:hscattered_MRD}). 
 \end{proof}

 \begin{proposition}
  There exist an MRD code $\C_4$ which is  $2$-MRD but neither $1$-MRD nor near MRD. In other words,
  $$\left(\A(q)\cap\A_{\mathrm{2}}(q)\right)\setminus \left(\A_{\mathrm{1}}(q)\cup \A_{\mathrm{N}}(q)\right)\neq \emptyset. $$
 \end{proposition}
  \begin{proof}
   Let $m\geq 6$ be even and let $\mU$ be a maximum scattered $[\frac{3m}{2},3]_{q^m/q}$ system. Since $3m/2>m$, any code  $\C_4 \in \Psi([\mU])$ cannot be $1$-MRD, due to  Corollaries \ref{cor:hscattered_sMRD} and \ref{cor:hscattered_bound}. Furthermore, since we also have that $3m/2>m+2$, by Theorem \ref{thm:nearMRD_characterization} the code $\C_4$ cannot be near MRD.
 \end{proof} 
 
 \begin{proposition}
  There exist a $2$-MRD code $\C_5$ which is  neither MRD nor near MRD. In other words,
  $$\A_{\mathrm{2}}(q)\setminus \left(\A(q)\cup \A_{\mathrm{N}}(q)\right)\neq \emptyset. $$
 \end{proposition}

\begin{proof}
 Let $m\geq 8$ be even and let $\mU$ be a maximum scattered $[\frac{3m}{2},3,m-1]_{q^m/q}$ system.  Let us choose any subspace $\mV$ that is an $[n,3,d]_{q^m/q}$ system, with $m+3 \leq  n < 3m/2$. Clearly, $\mV$ is still a scattered system and $d\leq m-1$. Hence, by Corollary \ref{cor:hscattered_sMRD}, any $\C_5 \in \Psi([\mV])$ is $2$-MRD. Furthermore, by Theorem \ref{thm:nearMRD_characterization} the code $\C_5$ cannot be near MRD, because of its length. Finally, if $d=m-1$ then we have
 $$km=3m>2n = n(m-d+1)=\min\{m(n-d+1),n(m-d+1)\}, $$ 
yielding a contradiction. Hence, we have $d\leq m-2$ and 
 \begin{align*}  km&=3m<3m+9 =3(m+3)\leq (m+3)(m-d+1)\leq n(m-d+1)\\
 &=\min\{m(n-d+1),n(m-d+1)\},\end{align*}
 and   $\C_5$ cannot be MRD.
\end{proof}

All these relations are graphically summarized in Figure \ref{fig:MRD}.

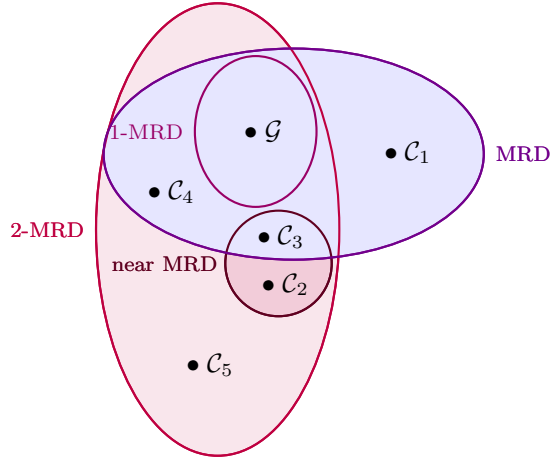
\begin{figure}[h!]
    \centering
    \begin{tikzpicture}[thick,auto,>=stealth,every edge/.append style={->},
    elli/.style={ellipse,draw,inner sep=0pt,
    simplify={#1},minimum width=2*\pgfkeysvalueof{/tikz/ellpar/a},
    minimum height=2*\pgfkeysvalueof{/tikz/ellpar/b}},
    simplify/.code={\tikzset{ellpar/.cd,#1}},
    ellpar/.cd,a/.initial=1cm,b/.initial=1cm]
 \path (-2,4) node[purple,elli={a=1.6cm,b=3cm},fill=purple!10,
    label={[purple]left:\scriptsize{2-MRD}}]{}
   (-1.5,5.3) node[purple!50,elli={a=0.8cm,b=1cm},fill=purple!15,
    label={[purple!50]left:\scriptsize{1-MRD}}](Omega){}
   (-1.2,3.55) node[purple!50!black,elli={a=0.7cm,b=0.7cm},fill=purple!20,
   label={[purple!50!black]left:\scriptsize{near MRD}\!}](chiB){}
   (-1,5) node[blue!40!purple,elli={a=2.5cm,b=1.4cm},fill=blue!10,
    label={[blue!40!purple]right:\scriptsize{MRD}}]{};
    \path (-2,4) node[purple,elli={a=1.6cm,b=3cm},
    label={[purple]left:\scriptsize{2-MRD}}]{}
   (-1.5,5.3) node[purple!80!blue,elli={a=0.8cm,b=1cm},
    label={[purple!80!blue]left:\scriptsize{1-MRD}}](Omega){}
   (-1.2,3.55) node[purple!50!black,elli={a=0.7cm,b=0.7cm},
   label={[purple!50!black]left:\scriptsize{near MRD}\!}](chiB){}
   (-1,5) node[blue!40!purple,elli={a=2.5cm,b=1.4cm},
    label={[blue!40!purple]right:\scriptsize{MRD}}]{};
    \path(-1.17,3.9) node(v) [circle] {\small{{\textbullet \;$\C_3$}}};
        \path(-2.61,4.5) node(v) [circle] {\small{{\textbullet \;$\C_4$}}};
        \path(-1.11,3.27) node(v) [circle] {\small{{\textbullet \;$\C_2$}}};
        \path(0.5,5.02) node(v) [circle] {\small{{\textbullet \;$\C_1$}}};
        \path(-2.1,2.2) node(v) [circle] {\small{{\textbullet \;$\C_5$}}};
        \path(-1.4,5.3) node(v) [circle] {\small{{\textbullet \;$\mathcal G$}}};

\end{tikzpicture}
    \caption{An illustration of the relations between the femilies of MRD, near MRD, $1$-MRD and $2$-MRD codes. The code indicated by $\mathcal G$ is a Delsarte-Gabidulin code.}
    \label{fig:MRD}
\end{figure}

\section{Quasi-maximum $h$-scattered subspaces and quasi-MRD codes}\label{sec:quasi_quasi}

When studying rank-metric codes, one may want to analyze codes with very good parameters, and try to get the maximum possible distance for the given dimension $k$, length $n$ and extension field degree $m$. Thus, one may want to rewrite the Singleton-like bound of \eqref{eq:singleton} as a pure upper bound on $d$, obtaining 
$$ d \leq \min \left\{ n-k+1,m-\frac{km}{n}+1 \right\}.$$
In doing so, we can see that the situation in which the minimum is attained by the fractional quantity, that is when $n>m$, cannot produce an MRD code if $n$ does not divide $km$. Hence, this motivated the study of $\Fmkd$ codes such that 
$$ d= m- \left\lceil \frac{km}{n}\right\rceil +1,$$
whose properties were hence investigated in \cite{de2018weight}. In this section we will study this class of codes,  analyze their geometric aspects and show their existence  using geometric arguments. 
\subsection{Definitions and correspondence}

More generally, we can link $h$-scattered subspaces with rank-metric codes with specific parameters. This is clear already from Theorem \ref{thm:correspondence_codes_systems} and Theorem \ref{thm:evasive_generalizedweights}. However, here we analyze this correspondence more in detail, also in terms of the rank defect of a rank-metric code. 

\begin{definition}
The \textbf{Singleton rank defect} (or simply \textbf{rank defect}) of an $\Fmkd$ code  $\C$  is given by
$$ \Rdef(\C)\coloneqq m-\left\lceil \frac{\dim_{\Fq}(\mC)}{n}\right\rceil-d+1=m-\left\lceil \frac{km}{n}\right\rceil-d+1.$$
A code is said to be \textbf{quasi-MRD} if $\Rdef(\C)=0$.
\end{definition}

Notice that for the definition of Singleton rank defect to be meaningful, it is needed that $n>m$, since it is based on the Singleton-like bound derived when $n\geq m$.  However, this assumption is not required for the truthfulness of what follows.

The following result relates $h$-scattered subspaces with rank-metric codes of a certain Singleton rank defect, and it is a generalization of Theorem \ref{thm:hscattered_MRD}.
\begin{theorem}\label{prop:rankdefect_hscattered}
 Let $\C$ be an $\Fmkd$ code and let $\mU\in\Phi([\C])$. Then, $\mU$ is $h$-scattered if and only if $\Rdef(\C^\perp)\leq\left\lfloor\frac{km}{n} \right\rfloor -h-1$.
\end{theorem}

\begin{proof}
 Recall that by Theorem \ref{thm:evasive_generalizedweights} we have that $\mU$ is $h$-scattered if and only if $\dd_{\rk}(\C^\perp)\geq h+2$. By definition of Singleton rank defect, we have that this is true if and only if
 $$\Rdef(\C^{\perp})=m-\left\lceil \frac{(n-k)m}{n}\right\rceil-\dd_{\rk}(\C^\perp)+1
 =\left\lfloor \frac{km}{n}\right\rfloor-\dd_{\rk}(\C^\perp)+1\leq\left\lfloor \frac{km}{n}\right\rfloor-h-1.  $$
\end{proof}

With this in mind, one can try to  study the existence of $h$-scattered subspaces  of maximum possible dimension when $h+1$ does not divide $km$, and relate them with the existence of special classes of rank-metric codes. We now describe the situation in which we want to work.

Assume that $h+1$ does not divide $km$ and let us write $km=(h+1)a+\epsilon$, with $1\leq \epsilon \leq h$. Then every $h$-scattered subspace/linear set has rank at most
$$n\leq \frac{km-\epsilon}{h+1}.$$

\begin{definition}
An $h$-scattered $\Fmk$ system is said to be \textbf{quasi-maximum} if its parameters meet  this bound with equality, that is, if its rank $n$ satisfies
$$ n=\left\lfloor\frac{km}{h+1} \right\rfloor =  \frac{km-\epsilon}{h+1},$$
where $\epsilon=km-(h+1)\lfloor\frac{km}{h+1} \rfloor$.
\end{definition}

\begin{remark}
We want to highlight the fact that the objects that we introduced as \emph{quasi-maximum} $h$-scattered systems have been usually called maximum $h$-scattered, like for the case of $h+1$ dividing $km$. However, we believe that the two cases are significantly different, as we will see in the rest of this section, and this is why we use this new name.
\end{remark}

We can immediately derive the following characterization of quasi-maximum $h$-scattered subspaces in terms of its associated code(s).

\begin{corollary}\label{cor:quasimaximum_rankdefect}
 Let $\mU$ be  an $\Fmk$ system and let $\C\in\Psi([\mU])$. Then, $\mU$ is quasi-maximum $h$-scattered if and only if $\Rdef(\C^\perp) \le \left\lfloor\frac{\epsilon}{a}\right\rfloor$, where $a$ and $\epsilon$ are the unique integers such that $km=a(h+1)+\epsilon$, with $0\leq \epsilon \leq h$. In other words, $a=\lfloor\frac{km}{h+1} \rfloor$ and  $\epsilon=km-(h+1)\lfloor\frac{km}{h+1} \rfloor$.
\end{corollary}

\begin{proof}
This is an immediate consequence of Theorem \ref{prop:rankdefect_hscattered}.
\end{proof}

As a consequence of Theorem \ref{prop:rankdefect_hscattered}, we derive a general result relating $h$-scattered $\Fmk$ system of maximum dimension  with quasi-MRD codes. The correspondence that we obtain is valid also when $h+1$ does not divide $km$, generalizing  Theorem \ref{thm:hscattered_MRD}.

\begin{theorem}[Correspondence quasi-maximum $h$-scattered subspaces -- quasi-MRD codes]\label{thm:quasimaximum_quasiMRD}
 Let $n=\lfloor\frac{km}{h+1}\rfloor$. Let $\mU$ be an $\Fmk$ system and $\C\in\Psi([\mU])$. Moreover, assume that 
 \begin{equation}\label{eq:condition_quasi} \left(km-(h+2)\left\lfloor\frac{km}{h+1}\right\rfloor\right)<0. \end{equation} 
 Then, $\mU$ is quasi-maximum $h$-scattered if and only if $\C^{\perp}$ is quasi-MRD.
\end{theorem}

\begin{proof}
 Let us consider the integers $a=\lfloor\frac{km}{h+1} \rfloor$ and  $\epsilon=km-(h+1)\lfloor\frac{km}{h+1} \rfloor$. Then we have
  $$ \left\lfloor\frac{\epsilon}{a}\right\rfloor=\left\lfloor \frac{km-(h+1)\left\lfloor \frac{km}{h+1}\right\rfloor}{\left\lfloor \frac{km}{h+1}\right\rfloor}\right\rfloor,$$
  which is equal to $0$ whenever 
  $$ \left(km-(h+2)\left\lfloor\frac{km}{h+1}\right\rfloor\right)<0. $$
  Thus, we can conclude by Corollary \ref{cor:quasimaximum_rankdefect} that $\mU$ is quasi-maximum $h$-scattered if and only if $\Rdef(\C^{\perp})=0$.
\end{proof}

Thus, in the range of parameters satisfying \eqref{eq:condition_quasi}, we have by Theorem \ref{thm:quasimaximum_quasiMRD} that
the existence of quasi-maximum $h$-scattered subspaces corresponds to the existence of $\Fm$-linear quasi-MRD codes.

\begin{remark}
 When $h+1$ divides $km$, the condition in \eqref{eq:condition_quasi} is always satisfied, and from Theorem \ref{thm:quasimaximum_quasiMRD}  we recover Theorem \ref{thm:hscattered_MRD}.
\end{remark}

\subsection{Existence of nontrivial quasi-MRD codes}

Here we consider the known construction of quasi-maximum $h$-scattered subspaces, which allow to show the existence of nontrivial quasi-MRD codes. 
Up to our knowledge, there are so far only two constructions of 
quasi-maximum $h$-scattered subspaces, both when $h=1$. 

The first construction is obtained by some  linear blocking sets of R\'edei type. In general, these are special scattered $[m+1,3]_{q^m/q}$ systems. For $m=3$, one obtains a quasi-maximum scattered $[4,3]_{q^3/q}$ systems. However, for any $\C\in \Psi([\mU])$, $\C^\perp$ is a $[4,1,3]_{q^3/q}$ code. Quasi-MRD codes with these parameters are not so interesting. 

Very recently, special constructions of quasi-maximum scattered $q$-systems have been found for some special values of $q$. More specifically, these constructions have only been found for characteristics $2,3$ and $5$. 

\begin{theorem}[\textnormal{\cite[Theorem 5.1]{bartoli2021evasive}}]\label{thm:qmaxscattered}
Let $h$ be a nonnegative integer. Consider $q=p^{15h+s}$, with $\gcd(s, 15) =1$ if $p=2,3$ and  with $s=1$ if $p=5$. Then there  exist a scattered $[7,3]_{q^5/q}$  system.
\end{theorem}

The proof of Theorem \ref{thm:qmaxscattered} is actually constructive, and its main idea is based on constructing a $q$-linearized polynomial over $\F_{q^{15}}$ of $q$-degree $7$, whose kernel has $\Fq$-dimension $7$ and that is scattered with respect to the $1$-dimensional $\F_{q^5}$-subspaces.

From a coding theoretic point of view, such $q$-systems give rise to nontrivial quasi-MRD codes.

\begin{corollary}
 Let $h$ be a nonnegative integer. Consider $q=p^{15h+s}$, with $\gcd(s, 15) =1$ if $p=2,3$ and  with $s=1$ if $p=5$. Then there  exist a $[7,4,3]_{q^5/q}$ quasi-MRD code. 
\end{corollary}

\begin{proof}
  Let $\mU$ be a maximum scattered $[7,3]_{q^5/q}$ system as in Theorem \ref{thm:qmaxscattered}. Since $k=3$ and $m=5$, we have that 
  $$ \left(km-(h+2)\left\lfloor\frac{km}{h+1}\right\rfloor\right) = 15- 21 <0. $$
  Thus, by Theorem \ref{thm:quasimaximum_quasiMRD}, we have that $\C^\perp$ is a $[7,4,3]_{q^5/q}$ quasi-MRD code
\end{proof}

\begin{remark}
Note that the existence of nontrivial $\Fq$-linear quasi-MRD codes was shown in \cite[Example 11]{de2018weight}. However, here we considered quasi-MRD code that have the additional property of being also $\Fm$-linear. By construction, those shown in \cite{de2018weight} are not $\Fm$-linear in general. 
\end{remark}

\section*{Acknowledgements}
The research of G. Marino and R. Trombetti was supported by the Italian National Group for Algebraic and Geometric Structures and their Applications (GNSAGA - INdAM). 

\bibliographystyle{abbrv}
\bibliography{biblio}

\end{document}